\documentclass[10pt,a4paper]{amsart}
\usepackage{amssymb}
\newtheorem{theorem}{Theorem}
\newtheorem{proposition}{Proposition}

\newtheorem{corollary}{Corollary}

\theoremstyle{remark}

\newtheorem*{remarks}{Remarks}

\newtheorem*{example}{Example}

\newtheorem*{definition}{Definition}

\newcommand{\remove}[1]{ }
\newcommand{\set}[1]{\left\{#1\right\}}

\def\NN{\mathbb N}

\def\uu_q{\mathcal{U}_q}
\def\vv_q{\mathcal{V}_q}
\def\uu{\mathcal{U}}
\def\uuu{\overline{\mathcal{U}}}

\begin{document}

\title{Optimal expansions in non-integer bases}
\author{Karma Dajani}
\address{Mathematics Department, Utrecht University, 3508 TA Utrecht, The Netherlands}
\email{k.dajani1@uu.nl}
\author{Martijn de Vries}
\address{Tussen de Grachten 213, 1381 DZ Weesp, The Netherlands}
\email{martijndevries0@gmail.com}
\author{Vilmos Komornik}
\address{D\'epartement de math\'ematique
       Universit\'e de Strasbourg
       7 rue Ren\'e Descartes
       67084 Strasbourg Cedex, France}
\email{vilmos.komornik@math.unistra.fr}
\author{Paola Loreti}
\address{Dipartimento di Scienze di Base e Applicate
per l'Ingegneria. Sezione di Matematica.
Sapienza  Universita' di Roma
Via A. Scarpa 16,  00161 Roma, Italy}
\email{loreti@dmmm.uniroma1.it}
\date{\today}
\thanks{Part of this work was done during the visit of the third author at the Department of Mathematics of the Delft
Technical University. He is grateful for this invitation and for the excellent working conditions.}

\begin{abstract}
For a given positive integer $m$, let $A=\set{0,1,\ldots,m}$ and $q \in (m,m+1)$. A sequence $(c_i)=c_1c_2 \ldots$ consisting of elements in $A$ is called an expansion of $x$ if $\sum_{i=1}^{\infty} c_i q^{-i}=x$. It is known that almost every $x$ belonging to the interval $[0,m/(q-1)]$ has uncountably many expansions. In this paper we study the existence of expansions $(d_i)$ of $x$ satisfying the inequalities
$\sum_{i=1}^n d_iq^{-i} \ge \sum_{i=1}^n c_i q^{-i}$ , $n=1,2,\ldots$ for each expansion $(c_i)$ of $x$.

\end{abstract}

\maketitle

\section{Introduction}
\label{s1}
Let $x \in [0,1)$. The decimal expansion
\begin{equation*}
x=\frac{b_1}{10}+\frac{b_2}{10^2}+\frac{b_3}{10^3}+\cdots ,
\end{equation*}
where we choose a finite expansion whenever it is possible, has a well known ``each-step'' optimality property: for each $k=1,2,\ldots ,$ among all finite sequences $c_1 \ldots c_k$ of integers with $0 \le c_i\le 9$ for $i=1,\ldots, k$, satisfying the inequality $\sum_{i=1}^k c_i 10^{-i}\le x,$ the sum $\sum_{i=1}^k b_i 10^{-i}$ is the closest to $x$. An analogous property holds for expansions in all integer bases $2,3,\ldots .$

In his celebrated paper \cite{Ren1957}, R\'enyi generalized these expansions to arbitrary real bases $q>1$ as follows. If $b_1,\ldots, b_{n-1}$ have already been defined for some $n\ge 1$ (no condition for $n=1$), then let $b_n$ be the largest integer satisfying the inequality
\begin{equation*}
\frac{b_1}{q}+\cdots+\frac{b_n}{q^n}\le x.
\end{equation*}
One may readily verify that
\begin{equation*}
\sum_{i=1}^{\infty}\frac{b_i}{q^i}=x;
\end{equation*}
it is called the \emph{greedy} expansion of $x$ in base $q$.

The purpose of this paper is to show that the natural analogue of the above optimality property fails for most non-integer bases, but it still holds for a particular countable set of bases, the smallest of them being the golden ratio $q=(1+\sqrt{5})/2\approx 1.618.$ Before formulating our result precisely we will first introduce expansions of real numbers with respect to a more general set of digits.

\medskip

Given a real number $q >1$ and a finite \emph{alphabet} or \emph{digit set} $A=\set{a_0,\ldots ,a_m}$ consisting of real numbers satisfying $a_0 < \cdots < a_m$, by an \emph{expansion} of $x$ (in \emph{base} $q$ with respect to $A$) we mean a sequence $(c_i)$ of \emph{digits} $c_i\in A$ satisfying
\begin{equation}\label{1}
\sum_{i=1}^{\infty} \frac{c_i}{q^i} =x.
\end{equation}
Pedicini \cite{Ped2005} proved the following basic result on the existence of such expansions.

\begin{proposition}\label{p1} Each $x\in J_{A,q}:= \left[a_0/(q-1),a_m/(q-1)\right]$ has an expansion if and only if
\begin{equation}\label{2}
\max_{1 \le j \le m} (a_j-a_{j-1}) \le \frac{a_m-a_0}{q-1}.
\end{equation}

\end{proposition}
For convenience of the reader we provide an elementary proof of this proposition. Observe that $(c_i)$ is an expansion of $x$ in base $q$ with respect to $A$ if and only if  $(c_i-a_0)=(c_1-a_0)(c_2-a_0) \ldots$ is an expansion of $x - a_0/(q-1)$ in base $q$ with respect to the alphabet
$\set{0, a_1-a_0, \ldots, a_m -a_0}$. Moreover, the inequality ~\eqref{2} holds if and only if the same inequality holds with $a_j-a_0$ in place of $a_j$, $0 \le j \le m$. Hence we may (and will) assume in the rest of this paper that $a_0=0$.

\begin{proof}[Proof of Proposition~\ref{p1}]
First assume that the inequality ~\eqref{2} holds. We define recursively a sequence $(b_i)$ with digits $b_i$ belonging to $A$ by applying the following \emph{greedy algorithm}: if for some integer $n\in \NN:=\set{1,2,\ldots}$ the digits $b_i$ have already been defined for all $1\le i<n$ (no condition for $n=1$), then let $b_n$ be the largest digit in $A$ satisfying the inequality $\sum_{i=1}^n b_i q^{-i} \le x$. Note that this algorithm is well defined for each $x \ge 0$. We show that $(b_i)$ is an expansion of $x$ for each $x$ belonging to $J_{A,q}$.

If $x = a_m/(q-1)$, then the greedy algorithm provides $b_i = a_m$ for all $i \ge 1$ whence $(b_i)$ is indeed an expansion of $x$.

If $0 \le x < a_m / (q-1)$, then there exists an index $n$ such that $b_n < a_m$. If $b_n < a_m$ for infinitely many $n$, then for each such $n$ we have
\begin{equation*}
0 \le x - \sum_{i=1}^{n} \frac{b_i}{q^i} < \frac{\max_{1 \le j \le m}(a_j-a_{j-1})}{q^n}.
\end{equation*}
Letting $n \to \infty$, we see that $(b_i)$ is an expansion of $x$.
Next we show that there cannot be finitely many $n$ such that $b_n < a_m$. Indeed, if there were a last index $n$ with $b_n = a_j < a_m$, then
\begin{equation*}
\left(\sum_{i=1}^{n} \frac{b_i}{q^i}\right) + \sum_{i=n+1}^{\infty} \frac{a_m}{q^i} \le x < \left(\sum_{i=1}^n \frac{b_i}{q^i}\right) + \frac{a_{j+1}-a_j}{q^n}
\end{equation*}
or equivalently
\begin{equation*}
\frac{a_m}{q-1} < a_{j+1} - a_j
\end{equation*}
contradicting \eqref{2}.

Finally, if the condition \eqref{2} does not hold, and $a_{\ell} - a_{\ell-1} > a_m/(q-1)$ for some $\ell \in \set{1, \ldots, m}$, then none of the numbers belonging to the nonempty interval
\begin{equation*}
\left(\frac{a_{\ell-1}}{q} + \sum_{i=2}^{\infty} \frac{a_m}{q^i} , \frac{a_{\ell}}{q} \right) \subset J_{A,q}
\end{equation*}
has an expansion.
\end{proof}

The proof of Proposition~\ref{p1} shows that if \eqref{2} holds, then each $x \in J_{A,q}$ has a lexicographically largest expansion $(b_i(x,A,q))$ which we call the \emph{greedy expansion} of $x$. The \emph{normalized errors} of an arbitrary expansion $(c_i)$ of $x$ are defined by
\begin{equation*}
\theta_n((c_i)):= q^n\left(x-\sum_{i=1}^n \frac{c_i}{q^i}\right), \quad n \in \NN.
\end{equation*}
We call an expansion $(d_i)$ of $x$ \emph{optimal} if $\theta_n((d_i)) \le \theta_n ((c_i))$ for each $n \in \NN$ and each expansion $(c_i)$ of $x$. It follows from the definitions that only the greedy expansion of a number $x \in J_{A,q}$ can be optimal. The following example shows that the greedy expansion of a number $x \in J_{A,q}$ is not always optimal. Other examples can be found in \cite{DajKra2002}.

\begin{example}
Let $A=\set{0,1}$ and $1 < q < (1+\sqrt{5})/2$. The sequence $(c_i):=011(0)^{\infty}$ is clearly an expansion of $x:=q^{-2}+q^{-3}$. Applying the greedy algorithm we find that the first three digits of the greedy expansion $(b_i)=(b_i(x,A,q))$ of $x$ equal $100$. Hence $\theta_3((b_i)) > \theta_3((c_i))=0$.
\end{example}

Let $A=\set{0,1 , \ldots,m}$ and $q \in (m,m+1)$ for some positive integer $m$. Proposition~\ref{p1} implies that in this case each $x \in J_{A,q}$ has an expansion. Let $P$ be the set consisting of those bases $q \in (m, m+1)$ which satisfy one of the equalities
\begin{equation*}
1= \frac{m}{q} + \cdots + \frac{m}{q^n} + \frac{p}{q^{n+1}}, \quad n \in \NN \mbox{ and } p \in \set{1, \ldots,m}.
\end{equation*}
We have the following dichotomy:

\begin{theorem}\label{t1}\mbox{}
\begin{itemize}
\item[\rm (i)] If $q \in P$, then each $x \in J_{A,q}$ has an optimal expansion.
\item[\rm (ii)] If $q \in (m,m+1) \setminus P$, then the set of numbers $x \in J_{A,q}$ with an optimal expansion is nowhere dense and has Lebesgue measure zero.
\end{itemize}
\end{theorem}

In Section~\ref{s2} we compare greedy expansions with respect to different alphabets. This gives us a characterization of optimal expansions which is essential to our proof of Theorem~\ref{t1} in Section~\ref{s3}. In Section~\ref{s4} we briefly discuss optimal expansions of real numbers in negative integer bases.
\section{Greedy expansions}
\label{s2}

Consider an alphabet $A=\set{a_0,a_1, \ldots,a_m}$ $(0=a_0 < \cdots < a_m)$ and a base $q$ satisfying the condition \eqref{2} as in the preceding section.
Let the \emph{greedy transformation} $T : J_{A,q} \to J_{A,q}$ corresponding to $(A,q)$ be given by
\begin{equation*}
T(x):=
\begin{cases}
qx-a_j&\text{if $x \in C(a_j):=\left[\frac{a_{j}}{q}, \frac{a_{j+1}}{q}\right)$, $0 \le  j < m$,}\\
qx-a_m&\text{if $x \in C(a_m):=\left[\frac{a_m}{q}, \frac{a_m}{q-1}\right]$.}
\end{cases}
\end{equation*}
Observe that $b_i(x,A,q)=a_j$ if and only if $T^{i-1}(x) \in C(a_j)$, $i \ge 1$.

For any fixed positive integer $k$, the equation \eqref{1} can be rewritten in the form
\begin{equation*}
\frac{d_1}{q^k}+\frac{d_2}{q^{2k}}+\cdots =x
\end{equation*}
by setting
\begin{equation*}
d_i:=\sum_{j=0}^{k-1}c_{ik-j}q^j,\quad i=1,2,\ldots .
\end{equation*}
In other words, every expansion in base $q$ with respect to the alphabet $A$ can also be considered as an expansion in base $q^k$ with respect to the alphabet
\begin{equation*}
A_k:=\set{c_1q^{k-1}+\cdots+c_k\ :\ c_1,\ldots, c_k\in A}\footnote{Other aspects of expansions with respect to alphabets of the form $A_k$ are studied in \cite{DarKat1988}, \cite{KomLor}.}.
\end{equation*}
(For $k=1$ this reduces to the original case.) In particular we have
\begin{equation*}
J_{A_k,q^k}=J_{A,q}
\end{equation*}
for every $k$. We may therefore compare the greedy transformation $T_k$ corresponding to $(A_k,q^k)$ with the $k$-th iteration $T^k$ of the map $T$ corresponding to $(A,q)$. It is easily seen that $T_k(x) \le T^k (x)$ for each $x \in J_{A,q}$ but in general we do not have equality here.

Given $(A,q)$ and a positive integer $k$, we denote by $S_{A,q,k}$ the set of sequences $(c_1,\ldots, c_k) \in A^k$ satisfying the following condition: if $(d_1,\ldots, d_k)\in A^k$ and $(d_1,\ldots, d_k)>(c_1,\ldots, c_k)$, then
\begin{equation*}
\sum_{i=1}^k\frac{d_i}{q^i}\ne \sum_{i=1}^k\frac{c_i}{q^i}.
\end{equation*}
For each $x \in J_{A,q}$, the sequence $b_1(x, A,q)\ldots b_k(x,A,q) 0^{\infty}$ is the greedy expansion in base $q$ with respect to $A$ of the number
\begin{equation*}
\sum_{i=1}^k \frac{b_i(x, A,q)}{q^i}
\end{equation*}
as follows from the definition of the greedy algorithm. Hence
\begin{equation*}
S_{A,q,k} \supset \set{(b_1(x,A,q),\ldots,b_k(x,A,q)): x \in J_{A,q}}.
\end{equation*}
Let the injective map $f: S_{A,q,k} \to J_{A,q}$ be given by
\begin{equation}\label{3}
f((c_1, \ldots , c_k))= \frac{c_1}{q} + \cdots + \frac{c_k}{q^k}, \quad (c_1, \ldots, c_k) \in S_{A,q,k}.
\end{equation}

\begin{proposition}\label{p2} The following statements are equivalent.
\begin{itemize}
\item[\rm(i)] The map $f$ is increasing.
\item[\rm (ii)] $T_k = T^k$.
\item[\rm (iii)] $S_{A,q,k} =  \set{(b_1(x,A,q),\ldots,b_k(x,A,q)): x \in J_{A,q}}$.
\end{itemize}
\end{proposition}

\begin{proof}
(i) $\Rightarrow$ (ii).
Given any $x\in J_{A,q}$, let $(c_1,\ldots, c_k)$ be the lexicographically largest sequence in $A^k$ satisfying
\begin{equation*}
s:=\frac{c_1}{q}+\cdots+\frac{c_k}{q^k}\le x.
\end{equation*}
Then $(c_1,\ldots, c_k) \in S_{A,q,k}$, and (i) implies that $T_k(x)=q^k(x-s)$. On the other hand, we also have
$T^k(x)=q^k(x-s)$ by definition of the greedy expansion.

(ii) $\Rightarrow$ (iii). Assume that $(c_1, \ldots ,c_k) \in S_{A,q,k}$, and let
\begin{equation*}
x':= \sum_{i=1}^k \frac{c_i}{q^i}.
\end{equation*}
If we had $(c_1, \ldots, c_k) \notin \set{(b_1(x,A,q),\ldots,b_k(x,A,q)): x \in J_{A,q}}$, then there would exist an index $m > k$ such that $b_m(x',A,q) \not=0$, whence $T_k(x')=0 < T^k(x')$, contradicting (ii).

(iii) $\Rightarrow$ (i). As already observed above, the sequence $b_1(x,A,q) \ldots b_k(x,A,q) 0^{\infty}$ is the greedy expansion of the number
\begin{equation*}
\sum_{i=1}^k \frac{b_i(x,A,q)}{q^i}.
\end{equation*}
It remains to note that $x < y$ if and only if $(b_i(x,A,q)) < (b_i(y,A,q))$ for numbers $x$ and $y$ belonging to $J_{A,q}$.
\end{proof}

\begin{remarks}\mbox{}
\begin{itemize}
\item[\rm (i)] Observe that the maps $T_k$ and $T^k$ are continuous from the right. Hence if $T_k \not = T^k$, then the maps $T_k$ and $T^k$ differ on a whole interval.
\item[\rm (ii)] If $T_k \not= T^k$, then $T_n \not= T^n$ for all $n \ge k$. In order to prove this, it is sufficient to show that $T_{k+1}\not=T^{k+1}$. By Proposition~\ref{p2} there exist two sequences $(b_1,\ldots,b_k)$, $(c_1,\ldots,c_k)$ both belonging to $S_{A,q,k}$ such that $(b_1,\ldots,b_k)<(c_1,\ldots,c_k)$, and
\begin{equation*}
\sum_{i=1}^k\frac{b_i}{q^i} > \sum_{i=1}^k\frac{c_i}{q^i}.
\end{equation*}
Note that the sequences
$(a_m,b_1,\ldots,b_k)$ and $(a_m,c_1,\ldots,c_k)$ both belong to $S_{A,q,k+1}$, and
\begin{equation*}
\frac{a_m}{q} + \sum_{i=1}^k\frac{b_i}{q^{i+1}} > \frac{a_m}{q} + \sum_{i=1}^k\frac{c_i}{q^{i+1}}.
\end{equation*}
Applying Proposition~\ref{p2} once more, we reach the desired conclusion.
\end{itemize}
\end{remarks}

\section{Proof of Theorem~\ref{t1}}
\label{s3}
 Let $m$ be a given positive integer. Throughout this section we consider expansions with respect to the alphabet $A=\set{0,1,\ldots, m}$ in a base $q$ belonging to $(m,m+1)$. For any integers $n\ge 1$ and $0\le p\le m$ we denote by $q_{m,n,p}$ the positive solution of the equation
\begin{equation*}
1=\frac{m}{q}+\cdots+\frac{m}{q^n}+\frac{p}{q^{n+1}}.
\end{equation*}
We have
\begin{equation*}
m=q_{m,1,0}<\cdots<q_{m,1,m}=q_{m,2,0}<\cdots<q_{m,2,m}=q_{m,3,0} <\cdots
\end{equation*}
and
\begin{equation*}
q_{m,n,p}\to\ m+1
\quad\text{if}\quad
n\to\infty .
\end{equation*}
Recall that the set $P$ introduced in Section~\ref{s1} consists of the numbers $q_{m,n,p}$ with $n \ge1$ and $1 \le p \le m$.
\begin{proposition}\label{p3} Let $n \ge 1$ and $ 1\le p \le m.$
\begin{itemize}
\item[\rm (i)] If $q=q_{m,n,p}$, then $T_k=T^k$ for all $k \ge 1$.

\item[\rm (ii)] If $q_{m,n,p-1}<q<q_{m,n,p}$, then $T_k=T^k$ if and only if $k \le n+1$.

\item[\rm (iii)] If $q \in (m,m+1) \setminus P$, then there exists a positive integer $k=k(q)$ such that the maps $T_k$ and $T^k$ differ on an interval contained in $[0,1)$.
\end{itemize}

\end{proposition}

\begin{proof} (i) By Proposition~\ref{p2} it is sufficient to prove that if
\begin{equation*}
(c_1, \ldots ,c_k),(d_1, \ldots, d_k)\in S_{A,q,k}
\quad\text{and}\quad (c_1, \ldots ,c_k) > (d_1, \ldots, d_k),
\end{equation*}
then
\begin{equation}\label{4}
\sum_{i=1}^{k} \frac{c_i}{q^i} > \sum_{i=1}^{k} \frac{d_i}{q^i}.
\end{equation}
Let $j$ be the first index such that $c_j>d_j$. Since $q=q_{m,n,p}$, the elements of $S_{A,q,k}$ do not contain any block of the form $am^nb$ with $a<m$ and $b\ge p$. Indeed, the sum corresponding to such a block is the same as the sum corresponding to the lexicographically larger block $(a+1)0^n(b-p)$. Therefore, since $d_j < m$, a block of the form $m^nb$ with $b \ge p$ cannot occur in $(d_{j+1}, \ldots,d_k)$. This implies that if $d_{\ell+1} \ldots d_{\ell +n+1}$ is a block of length $n+1$ that is contained in $(d_{j+1}, \ldots,d_k)$, then
\begin{align*}
\sum_{i=1}^{n+1} \frac{d_{\ell+i}}{q^i} &\le \max \set{\frac{m}{q} + \cdots + \frac{m}{q^{n-1}} + \frac{m-1}{q^n} + \frac{m}{q^{n+1}}, \frac{m}{q} + \cdots + \frac{m}{q^n} + \frac{p-1}{q^{n+1}}}\\
&=\frac{m}{q} + \cdots + \frac{m}{q^n} + \frac{p-1}{q^{n+1}}.
\end{align*}
Therefore
\begin{equation*}
\sum_{i=j+1}^{k} \frac{d_i}{q^i} <
\frac{1}{q^j}\sum_{k=0}^{\infty} \left(\frac{1}{q^{n+1}}\right)^k \left(\frac{m}{q} + \cdots + \frac{m}{q^n} + \frac{p-1}{q^{n+1}}\right)
=\frac{1}{q^j}
\end{equation*}
which implies \eqref{4}.

(ii) It follows from our assumption on $q$ that
\begin{equation}\label{5}
\frac{m}{q^2}+\cdots+\frac{m}{q^{n+1}}+\frac{p-1}{q^{n+2}}
<\frac{1}{q}<
\frac{m}{q^2}+\cdots+\frac{m}{q^{n+1}}+\frac{p}{q^{n+2}}.
\end{equation}

First we show that $T_k=T^k$ for every $k\le n+1$. Let $(c_1,\ldots ,c_k)$ and $(d_1,\ldots ,d_k)$ be sequences in $A^k$ satisfying $(c_1,\ldots,c_k) > (d_1,\ldots,d_k)$, and let $j$ be the smallest positive integer such that $c_j>d_j$. Then we have
\begin{align*}
 \sum_{i=1}^k \frac{c_i-d_i}{q^i}
&\ge\frac{1}{q^{j-1}}\left(\frac{1}{q}-\frac{m}{q^2}-\cdots -\frac{m}{q^{k+1-j}}\right)\\
&\ge \frac{1}{q^{j-1}}\left(\frac{1}{q}-\frac{m}{q^2}-\cdots -\frac{m}{q^{n+1}}\right)\\
&>0
\end{align*}
by using \eqref{5} in the last step.

Due to a remark following the proof of Proposition~\ref{p2}  it remains to show that $T_{n+2}\ne T^{n+2}$. The sequence $10^{n+1}$ clearly belongs to $S_{A,q,n+2}$. In order to show that $0m^np$ belongs to $S_{A,q,n+2}$ as well, we must prove that
\begin{equation*}
\sum_{i=1}^{n+2} \frac{c_i}{q^i}\ne \frac{m}{q^2}+\cdots+\frac{m}{q^{n+1}}+\frac{p}{q^{n+2}}
\end{equation*}
for every sequence $c_1\ldots c_{n+2}\in A^{n+2}$ satisfying $c_1\ldots c_{n+2}>0m^np$.

If $c_1=0$, this is clear. If $c_1\ldots c_{n+2}=10^{n+1}$, then
\begin{equation}\label{6}
\sum_{i=1}^{n+2} \frac{c_i}{q^i}=\frac{1}{q}<\frac{m}{q^2}+\cdots+\frac{m}{q^{n+1}}+\frac{p}{q^{n+2}}
\end{equation}
by \eqref{5}. In the remaining cases we have $c_1\ge 1$ and $c_1+\cdots +c_{n+2}\ge 2$,  so that
\begin{equation}\label{7}
\sum_{i=1}^{n+2} \frac{c_i}{q^i}\ge \frac{1}{q}+\frac{1}{q^{n+2}}>\frac{m}{q^2}+\cdots+\frac{m}{q^{n+1}}+\frac{p}{q^{n+2}}
\end{equation}
by \eqref{5} again.

Since $10^{n+1}, 0m^np\in S_{A,q,n+2}$ and $10^{n+1}>0m^np$, the inequality \eqref{6} shows that the map \eqref{3} with $k=n+2$ is not increasing.

(iii) As in part (ii), suppose that $q_{m,n,p-1}< q < q_{m,n,p}$ for some $n,p \ge 1$. It follows from \eqref{6} and \eqref{7} that if $x$ belongs to the nonempty interval
\begin{equation*}
D:=\left[\frac{m}{q^2}+\cdots+\frac{m}{q^{n+1}}+\frac{p}{q^{n+2}}, \frac{1}{q}+\frac{1}{q^{n+2}}\right),
\end{equation*}
then
\begin{equation*}
\sum_{i=1}^{n+2} \frac{b_i(x,A,q)}{q^i}=\frac{1}{q} <  \frac{m}{q^2}+\cdots+\frac{m}{q^{n+1}}+\frac{p}{q^{n+2}}=\frac{b_1(x,A_{n+2},q^{n+2})}{q^{n+2}},
\end{equation*}
i.e.,

\begin{equation*}
T_{n+2}(x) = q^{n+2}\left(x-\frac{m}{q^2}-\cdots-\frac{m}{q^{n+1}}-\frac{p}{q^{n+2}}\right) < q^{n+2}\left(x-\frac{1}{q}\right)=T^{n+2}(x).
\end{equation*}
If $(m,n,p)\not=(1,1,1)$ then the interval $D$ is contained in $[0,1)$. If $(m,n,p)=(1,1,1)$ and $1 > q^{-2} + q^{-3}$, then $D \cap [0,1)$ is nonempty. Therefore, also in this case the maps $T_{n+2}$ and $T^{n+2}$ differ on an interval contained in $[0,1)$. It remains to consider those values of $q$ that satisfy $1 \le q^{-2} + q^{-3}$.

If $1 \le q^{-2} + q^{-3}$, then let $\ell \ge 3$ be the (unique) positive integer satisfying
\begin{equation}\label{8}
\frac{1}{q^{\ell}} + \frac{1}{q^{\ell+1}} <  1 \le \frac{1}{q^{\ell-1}} + \frac{1}{q^{\ell}}.
\end{equation}
If the latter inequality in \eqref{8} is strict, then for each $x$ belonging to the nonempty interval
\begin{equation*}
\left[\frac{1}{q^{\ell}} + \frac{1}{q^{\ell+1}},\min \set{1,\frac{1}{q}+ \frac{1}{q^{\ell+1}}} \right),
\end{equation*}
we have $b_1(x,A,q) \ldots b_{\ell+1}(x,A,q)=10^{\ell}$, and
\begin{equation*}
T_{\ell+1}(x) \le q^{\ell+1} \left(x - \frac{1}{q^{\ell}} - \frac{1}{q^{\ell+1}}\right)< q^{\ell+1}\left(x-\frac{1}{q}\right)=T^{\ell+1}(x).
\end{equation*}
If the latter inequality in $\eqref{8}$ is in fact an equality, then we consider the nonempty interval
\begin{equation*}
\left[\frac{1}{q^{\ell-1}} + \frac{1}{q^{\ell+1}},\min \set{1,\frac{1}{q}+ \frac{1}{q^{\ell+1}}} \right).
\end{equation*}
For each $x$ belonging to this interval we have $b_1(x,A,q) \ldots b_{\ell+1}(x,A,q)=10^{\ell}$, and
\begin{equation*}
T_{\ell+1}(x) \le q^{\ell+1} \left(x - \frac{1}{q^{\ell-1}} - \frac{1}{q^{\ell+1}}\right)< q^{\ell+1}\left(x-\frac{1}{q}\right)=T^{\ell+1}(x).
\end{equation*}

For each $q \in (m,m+1) \setminus P$ we now have constructed an interval $I \subset [0,1)$ and a positive integer $k$ such that $T_k < T^k$ on $I$.
\end{proof}
\begin{remarks}\mbox{}
\begin{itemize}
\item[\rm (i)] It follows from the above proof that if $q_{m,n,p-1} < q < q_{m,n,p}$ $(n,p \ge 1)$ and $(m,n,p) \not= (1,1,1)$, then one may take $k=n+2$ in the statement of Proposition~\ref{p3}(iii).
\item[\rm (ii)] If $T_k(x) \not= T^k(x)$ for some $x \in [0,1)$, then the first digit of any expansion of $xq^{-1}$ in base $q$ with respect to $A$ must be zero, whence
\begin{equation*}
T_{k+1}\left(\frac{x}{q}\right)=T_k(x)< T^k(x) =T^{k+1}\left(\frac{x}{q}\right).
\end{equation*}
Hence if $T_k \not= T^k$ on a subinterval of $[0,1)$, then $T_n \not= T^n$ on a subinterval of $[0,1)$ for each integer $n \ge k$.
\end{itemize}
\end{remarks}

\begin{proof}[Proof of Theorem~\ref{t1}.] (i) Let $q \in P$. Note that the greedy expansion of $x \in J_{A,q}$ is optimal if and only if $T_k(x)=T^k(x)$ for each $k \ge 1$. Hence each $x \in J_{A,q}$ has an optimal expansion by Proposition~\ref{p3}(i).

(ii) Let $q \in (m,m+1) \setminus P$. It is well known (see, e.g., \cite{Par1960}, \cite{Ren1957}) that the map $T$ is ergodic with respect to a
unique normalized absolutely continuous $T$-invariant measure $\mu$ with a density that is positive on the interval $[0,1)$. According to Proposition~\ref{p3}(iii) there exists an interval $I \subset [0,1)$ and a number $k=k(q)$ such that $T_k < T^k$ on $I$. An application of Birkhoff's ergodic theorem yields that for almost every $x \in [0,1)$ there exists a positive integer $\ell=\ell(x)$ such that $T^{\ell}(x) \in I$. For each such $x$ the greedy expansion of $x$ is not optimal because the greedy expansion $b_{\ell+1}(x,A,q) b_{\ell+2}(x,A,q)\ldots$ of $T^{\ell}(x)$ is not optimal.
Since the map $T$ is nonsingular \footnote{Nonsingularity of $T$ means that $T^{-1}(B)$ is a null set whenever $B \subset J_{A,q}$ is a null set.} and since for each $x \in [1,m/(q-1))$ there exists a positive integer $n=n(x)$ such that $T^n(x) \in [0,1)$, we may conclude that $x$ has no optimal expansion for almost every $x \in J_{A,q}$.

It remains to show that that the set of numbers with an optimal expansion is nowhere dense. We call an expansion $(d_i)$ of a number $x \in J_{A,q}$ \emph{infinite} if $d_n>0$ for infinitely many $n \in \NN$. Otherwise it is called \emph{finite}. Let $x \in J_{A,q}$ be a number with no optimal and no finite expansion, and let $(b_i)=(b_i(x,A,q))$. Then there exists an expansion $(c_i)$ of $x$ and a number $n \in \NN$ such that the inequalities
\begin{equation*}
\sum_{i=1}^n \frac{b_i}{q^i} < \sum_{i=1}^n \frac{c_i}{q^i} < x
\end{equation*}
hold. Hence the number $x$ belongs to the interior of the interval
\begin{equation*}
E:=\left[\sum_{i=1}^n \frac{c_i}{q^i} , \left(\sum_{i=1}^n \frac{c_i}{q^i}\right) + \sum_{i=n+1}^{\infty} \frac{m}{q^i}\right].
\end{equation*}
It follows from Proposition~\ref{p1} that the set $E$ consists precisely of those numbers in $J_{A,q}$ that have an expansion starting with $c_1 \ldots c_n$. Since $(b_i)$ is infinite by hypothesis, there exists a number $\delta=\delta(x) >0$ such that $(x-\delta, x+\delta) \subset E$ and such that the greedy expansion of each number belonging to $(x-\delta,x+\delta)$ starts with $b_1 \ldots b_n$ (this follows for instance from Lemmas 3.1 and 3.2 in \cite{DVK2009}). Hence none of the numbers in $(x-\delta,x+\delta)$ has an  optimal expansion. Denoting by $\mathcal{O}_q$ the set of numbers in $J_{A,q}$ with an optimal expansion and its closure by $\overline{\mathcal{O}_q}$ we may thus conclude that numbers belonging to $\overline{\mathcal{O}_q} \setminus \mathcal{O}_q$ have a finite expansion whence $\overline{\mathcal{O}_q} \setminus \mathcal{O}_q$ is at most countable. This implies in particular that the set $\overline{\mathcal{O}_q}$ is also a null set and has therefore no interior points.
\end{proof}

For each positive integer $k$, the map $T_k$ is also ergodic with respect to a unique normalized absolutely continuous $T_k$-invariant measure $\mu_k$ as follows from Theorem 4 in \cite{LY1982}. Since $T_1=T$, the measure $\mu$ introduced in the proof of Theorem~\ref{t1} equals $\mu_1$. Methods to construct an explicit formula for (a version of) the density of the measure $\mu_k$ can be found in \cite{K1990} (see also \cite{G2009}, \cite{DK2009}).

\begin{corollary}\label{c1} $q \in P$ if and only if $\mu_1=\mu_k$ for each $k \ge1$.
\end{corollary}
\begin{proof}
Proposition~\ref{p3}(i) implies that $\mu_1=\mu_2 = \cdots$ if $q$ belongs to $P$. Conversely, suppose that $q \in (m,m+1) \setminus P$ and let $I \subset [0,1)$ be an interval such that $T_k < T^k$ on $I$ for some positive integer $k$. Since the maps $T_k$ and $T^k$ are continuous from the right, there exists a subinterval $J \subset I$ and a number $t>0$ such that $T_k < t < T^k$ on $J$. Note that $T^{-k}([0,t)) \subset T_k^{-1}([0,t))$ because $T_k \le T^k$ on $J_{A,q}$. If we had $\mu_k = \mu_1$, then $\mu_1$ would also be $T_k$-invariant, whence
\begin{equation*}
0=\mu_1\left(T_k^{-1}[0,t)\right) - \mu_1\left(T^{-k}[0,t)\right) \ge \mu_1(J)
\end{equation*}
which contradicts the fact that the density of $\mu_1$ is positive on the interval $[0,1)$.
\end{proof}
\begin{remarks}\mbox{}
\begin{itemize}
\item[\rm (i)] For \emph{each} $q \in (m,m+1)$, almost every $x \in J_{A,q}$ has uncountably many expansions (see \cite{Si2003}, \cite{DDV2007}). It follows from Theorem~\ref{t1}(i) that a number with an optimal expansion may have uncountably many expansions. We do not know whether the greedy expansion of a number with at most countably many expansions is always optimal.
\item[\rm (ii)] It has been shown in \cite{GS2001} (see also \cite{DVK2009}, \cite{DVK2010}) that if $q\in (m,m+1)$ is close enough to $m+1$, then the set $\uu_q$ of numbers in $J_{A,q}$ with a unique expansion is uncountable. Moreover, the Hausdorff dimension of $\uu_q$ tends to one if $q \to m+1$. Since a unique expansion is clearly optimal, the same properties hold for the set of numbers belonging to $J_{A,q}$ with an optimal expansion.
\item[\rm (iii)] Let $\uu$ be the set of bases $q \in (m,m+1)$ such that the number $1 \in J_{A,q}$ has a unique expansion. The set $\uu$ has been extensively studied in \cite{EHJ1991}, \cite{KomLor2007}, \cite{DVK2009}. For instance it has been shown in \cite{DVK2009} that $\uu_q$ is closed if and only if $q \in (m,m+1) \setminus \uuu$ where $\uuu$ is the closure of $\uu$. It follows from the proof of Theorem 1.3 in \cite{DVK2009} that each number $x$ belonging to the closure $\overline{\uu_q}$ of the set $\uu_q$ has an optimal expansion for each $q \in (m,m+1)$. We conclude this section with an example showing that the set $\mathcal{O}_q$ of numbers with an optimal expansion properly contains $\overline{\uu_q}$ for all $q \in (m,m+1)$.
\end{itemize}
\end{remarks}
\begin{example} Fix $q \in (m,m+1)$. It is well known that each number $x \in J_{A,q} \setminus \set{0}$ has a lexicographically largest infinite expansion $(a_i(x))$ which coincides with its greedy expansion if and only if the latter is infinite. If the greedy expansion $(b_i(x))$ of a number $x \in J_{A,q}\setminus \set{0}$ is finite and $b_n(x)$ is its last nonzero element, then $(a_i(x))=b_1(x) \ldots b_{n-1}(x) (b_n(x)-1) a_1(1) a_2(1) \ldots$. For convenience we set $(a_i(0)):=0^{\infty}$. It is shown in \cite{DVK2009} that $\overline{\uu_q} \subset \vv_q$ where $\vv_q$ is the set of numbers $x \in J_{A,q}$ such that
\begin{equation*}
(m-a_{n+1}(x))(m-a_{n+2}(x)) \ldots \le a_1(1) a_2(1) \ldots \quad \text{whenever} \quad a_n(x) >0.
\end{equation*}
Let $k$ be the largest positive integer satisfying the inequality $\sum_{i=1}^k mq^{-i} < 1$, and consider the number
\begin{equation*}
x:=\frac{1}{q} + \frac{1}{q^{k+2}}.
\end{equation*}
The greedy expansion $(b_i(x))$ of $x$ is clearly given by $10^k10^{\infty}$. Our choice of $k$ implies that $(b_i(x))$ is optimal. However, the number $x$ does not belong to $\vv_q$ because $a_1(x) \ldots a_{k+2}(x)=10^{k+1}$ and $a_1(1) \ldots a_{k+1}(1)=m^k c$ with $c < m$.
\end{example}

\section{Optimal expansions in negative bases}\label{s4}

Given a positive integer $m$ and a real number $m <q\le m+1$, by an expansion of a real number $x$ in base $-q$ we mean a sequence $(c_i)=c_1c_2 \ldots$ of integers $c_i\in A:=\set{0,1,\ldots,m}$ satisfying
\begin{equation*}
\sum_{i=1}^{\infty} \frac{c_i}{(-q)^i} =x.
\end{equation*}

One easily verifies that $(c_i)$ is an expansion of a real number $x$ in base $-q$ if and only if $(c_i'):=(m-c_1,c_2,m-c_3,c_4,\ldots)$ is an expansion of $x':=x+mq/(q^2-1)$ in base $q$ (with respect to $A$). It follows from Proposition~\ref{p1} that each $x$ belonging to the interval
\begin{equation*}
J_{A,-q}:=\left[\frac{-mq}{q^2-1},\frac{m}{q^2-1}\right]
\end{equation*}
has an expansion in base $-q$.

\begin{definition}
An expansion $(d_i)$ of $x$ in base $-q$ is \emph{optimal} if for any other expansion $(c_i)$ of $x$ in base $-q$ we have
\begin{equation*}
\left| x-\sum_{i=1}^{n} \frac{d_i}{(-q)^i}\right|\le \left| x-\sum_{i=1}^{n} \frac{c_i}{(-q)^i}\right|
\end{equation*}
for all $n=1,2,\ldots .$
\end{definition}

We only consider here expansions in negative integer bases $-2,-3,\ldots$. While in positive integer bases the greedy expansion is always optimal, in negative integer bases there are infinitely many numbers with no optimal expansion:

\begin{proposition}\label{p4}
In negative integer bases only the unique expansions are optimal.
\end{proposition}

\begin{proof}
Let $q=m+1$ for some positive integer $m$. If $x \in J_{A,-q}$ has no unique expansion in base $-q$  then $x$ has exactly two expansions $(c_i)$ and $(d_i)$ in base $-q$ because $(c_i')$ and $(d_i')$ are the only expansions of $x'$ in base $q$. Moreover, there exists a positive integer $k$ such that $c_i'=d_i'$ for $1 \le i \le k-1$ and such that the sequences $(c_k',c_{k+1}', \ldots )$ and $(d_k',d_{k+1}',\ldots)$ are equal to $(p+1)0^{\infty}$ or $pm^{\infty}$ for some $p\in \set{0,\ldots, m-1}$. If necessary, interchange $(c_i)$ and $(d_i)$ so that $(c_i') > (d_i')$, and let $n$ be a positive integer such that $2n \ge k$. Then
\begin{equation*}
x=\left( \sum_{i=1}^{2n}\frac{c_i}{(-q)^i}\right)-\sum_{i=n}^{\infty}\frac{m}{q^{2i+1}}=\left( \sum_{i=1}^{2n}\frac{d_i}{(-q)^i}\right)+\sum_{i=n}^{\infty}\frac{m}{q^{2i+2}}
\end{equation*}
whence
\begin{equation*}
\left| x-\sum_{i=1}^{2n+1} \frac{c_i}{(-q)^i}\right|=\frac{1}{q}\left| x-\sum_{i=1}^{2n+1} \frac{d_i}{(-q)^i}\right|<\left| x-\sum_{i=1}^{2n+1} \frac{d_i}{(-q)^i}\right|,
\end{equation*}
and
\begin{equation*}
\left| x-\sum_{i=1}^{2n} \frac{d_i}{(-q)^i}\right|=\frac{1}{q}\left| x-\sum_{i=1}^{2n} \frac{c_i}{(-q)^i}\right|<\left| x-\sum_{i=1}^{2n} \frac{c_i}{(-q)^i}\right|
\end{equation*}
so that the expansions $(c_i)$ and $(d_i)$ are not optimal.
\end{proof}


\begin{thebibliography}{[K$^3$1]}

\bibitem{DDV2007} K. Dajani, M. de Vries,
\emph{Invariant densities for random $\beta$-expansions},
J. Eur. Math. Soc. {\bf 9} (2007), 157--176.

\bibitem{DK2009} K. Dajani, C. Kalle,
\emph{A natural extension for the greedy $\beta$-transformation with three arbitrary digits},
Acta Math. Hungar. {\bf 125} (2009), 21--45.

\bibitem{DajKra2002} K. Dajani, C. Kraaikamp,
\emph{From greedy to lazy expansions and their driving dynamics},
Expo. Math. {\bf 20} (2002), 315--327.

\bibitem{DarKat1988} Z. Dar\'oczy, I. K\'atai,
{\em Generalized number systems in the complex plane},
Acta Math. Hungar. {\bf 51} (1988), 409--416.

\bibitem{DVK2009} M. de Vries, V. Komornik,
\emph{Unique expansions of real numbers},
Adv. Math. {\bf 221} (2009), 390--427.

\bibitem{DVK2010} M. de Vries, V. Komornik,
\emph{A two-dimensional univoque set},
Fund. Math, in press.

\bibitem{EHJ1991}
P. Erd\H os, M. Horv\'ath, I. Jo\'o,
\emph{On the uniqueness of the expansions} $1=\sum q^{-n_i}$,
Acta Math. Hungar. {\bf 58} (1991), 333-342.

\bibitem{GS2001} P. Glendinning, N. Sidorov,
\emph{Unique representations of real numbers in non-integer bases},
Math. Res. Lett. {\bf 8} (2001), 535--543.

\bibitem{G2009} P. G\'ora,
\emph{Invariant densities for piecewise linear maps of the unit interval},
Ergodic Theory Dynam. Systems {\bf 29} (2009), 1549--1583.

\bibitem{KomLor2007} V. Komornik, P. Loreti,
\emph{On the  topological structure of univoque sets},
J. Number Theory {\bf 122} (2007), 157--183.

\bibitem{KomLor}  V. Komornik, P. Loreti,
\emph{Universal expansions in negative and complex bases},
Integers {\bf 10} (2010), 669--679.

\bibitem{K1990} C. Kopf,
\emph{Invariant measures for piecewise linear transformations of the interval},
Appl. Math. Comput. {\bf 39} (1990), 123-144.

\bibitem{LY1982} A. Lasota, J. A. Yorke,
\emph{Exact dynamical systems and the Frobenius-Perron operator},
Trans. Amer. Math. Soc. {\bf 273}, 375--384.

\bibitem{Par1960} W. Parry,
\emph{On the $\beta $-expansions of real numbers},
Acta Math. Acad. Sci. Hungar. {\bf 11} (1960), 401--416.

\bibitem{Ped2005} M. Pedicini,
\emph{Greedy expansions and sets with deleted digits},
Theoret. Comput. Sci. {\bf 332} (2005), 313--336.

\bibitem{Ren1957} A. R\'{e}nyi,
\emph{Representations for real numbers and their ergodic properties},
Acta Math. Acad. Sci.\ Hungar. {\bf 8} (1957), 477-493.

\bibitem{Si2003} N. Sidorov,
\emph{Almost every number has a continuum of $\beta$-expansions},
Amer. Math. Monthly {\bf 110} (2003), 838--842.

\end{thebibliography}
\end{document}